\documentclass[11pt]{amsart}
\usepackage{amsfonts,amssymb,amsmath,amsthm}
\usepackage[english]{babel}
\usepackage{mathrsfs}
\usepackage{geometry}
\usepackage{enumerate}
\usepackage{graphicx,color}

\theoremstyle{plain}

\theoremstyle{plain}
\newtheorem{theorem}{Theorem} [section]

\newtheorem{lemma}[theorem]{Lemma}

\theoremstyle{definition}

\newtheorem{remark}[theorem]{Remark}
\numberwithin{theorem}{section}
\numberwithin{equation}{section}
\numberwithin{figure}{section}

\def\mean#1{\mathchoice
         {\mathop{\kern 0.2em\vrule width 0.6em height 0.69678ex depth -0.58065ex
                 \kern -0.8em \intop}\nolimits_{\kern -0.4em#1}}%
         {\mathop{\kern 0.1em\vrule width 0.5em height 0.69678ex depth -0.60387ex
                 \kern -0.6em \intop}\nolimits_{#1}}%
         {\mathop{\kern 0.1em\vrule width 0.5em height 0.69678ex
             depth -0.60387ex
                 \kern -0.6em \intop}\nolimits_{#1}}%
         {\mathop{\kern 0.1em\vrule width 0.5em height 0.69678ex depth -0.60387ex
                 \kern -0.6em \intop}\nolimits_{#1}}}

\def\N{\mathbb N}

\def\R{\mathbb R}

\def\e{\varepsilon}

\def\n{\nabla}

\DeclareMathOperator{\diam}{diam}

\DeclareMathOperator{\Id}{Id}

\title[]{Rigidity and stability of  Caffarelli's log-concave perturbation theorem}

\author[G. De Philippis]{Guido De Philippis}
\address{SISSA, Via Bonomea 265, 34136 Trieste, Italy.}
\email{guido.dephilippis@sissa.it}

\author[A. Figalli]{Alessio Figalli}
\address{Department of Mathematics,
The University of Texas at Austin, 2515 Speedway, RLM 8.100 Stop C1200,
Austin TX 78712, USA.}
\email{figalli@math.utexas.edu}


\begin{document}

\begin{abstract}
In this note we establish some rigidity and stability results for  Caffarelli's log-concave perturbation theorem. As an application we show that if  a \(1\)-log-concave measure has almost the same Poincar\'e constant as the Gaussian measure, then it almost splits off a Gaussian factor.
\end{abstract}

 \dedicatory{To Nicola Fusco, for his 60th birthday, con affetto e ammirazione.}

\maketitle

\section{Introduction}

Let $\gamma_n$ denote the centered Gaussian measure in \(\R^n\), i.e. \(\gamma_n=(2\pi)^{-n/2}e^{-|x|^2/2}dx\), and let $\mu$ be a probability measure
on $\R^n$.
By a  classical theorem of Brenier \cite{Br}, there exists a convex function $\varphi:\R^n\to \R$ such that $T=\nabla\varphi:\R^n\to\R^n$ transports $\gamma_n$ onto $\mu$, i.e. \(T_\sharp \gamma_n=\mu\), or equivalently
\begin{equation*}
\int h\circ T \,d\gamma_n=\int h\, d\mu\qquad \textrm{for all continuous and bounded functions \(h\in C_b(\R^n)\)}.
\end{equation*}
In the sequel we will refer to \(T\) as the {\em Brenier map} from $\gamma_n$ to $\mu$.

In \cite{Caf1,Caf2} Caffarelli proved that if $\mu$ is ``more log-concave'' than $\gamma_n$, then $T$ is $1$-Lipschitz, that is, all the eigenvalues 
of $D^2\varphi$ are bounded from above by $1$. Here is the exact statement:

\begin{theorem}[Caffarelli]\label{thm:caffarelli} Let \(\gamma_n\) be the Gaussian measure in \(\R^n\), and let \(\mu=e^{-V} dx\) be a probability measure satisfying \(D^2 V\ge \Id_n\). Consider the Brenier  map \(T=\nabla \varphi\) from \(\gamma_n\) to \(\mu\). Then \(T\)  is \(\)1-Lipschitz, i.e. \(D^2\varphi(x)\le\Id \)
for a.e. $x$.
\end{theorem}

This theorem allows one to show that optimal constants in several functional inequalities are extremized by the Gaussian measure. More precisely, let \(F,G,H,L,J\) be continuous  functions on \(\R\)  and assume that \(F,G,H,J\)  are nonnegative, and that \(H\) and \(J\) are {increasing}. For \(\ell\in \R_+\) let 
\begin{equation}\label{eq:lambda}
\lambda(\mu,\ell):=\inf\Bigg\{\frac{H\Big(\int J(|\nabla u|) \,d\mu\Big)} {F\Big(\int G(u)\, d\mu\Big)}\,:\qquad u\in {\rm Lip}(\R^n)\,, \int L(u)\,d\mu=\ell\Bigg\}.
\end{equation}
Then
 \begin{equation}\label{eq:lambda2}
 \lambda(\gamma_n,\ell)\le \lambda(\mu,\ell).
 \end{equation}
Indeed, given a function \(u\) admissible in the variational formulation for \(\mu\), we set \(v:=u\circ T\)  and note that, since \(T_\sharp \gamma_n=\mu\),
$$
\int K(v)\,d\gamma_n=\int K(u\circ T)\,d\gamma_n=\int K(u)\,d\mu\qquad \text{for $K=G,L$.}
$$
In particular, this implies that $v$ is admissible in the variational formulation for \(\gamma_n\). 
Also, thanks to Caffarelli's Theorem, 
\[
|\nabla v|\le |\nabla u|\circ T\,|\nabla T|\le  |\nabla u|\circ T,
\]
therefore
$$
H\Big(\int J(|\nabla v|)\, d\gamma_n\Big)
\leq H\Big(\int J(|\nabla u|)\circ T\, d\gamma_n\Big)=
H\Big(\int J(|\nabla u|)\, d\mu\Big).
$$
Thanks to these formulas, \eqref{eq:lambda2} follows easily.

Note that the classical Poincar\'e and Log-Sobolev inequalities fall in the above general framework. \\

Two  questions that naturally arise from the above considerations are:
\begin{itemize}
\item[-]\emph{Rigidity}: What can be said of \(\mu\) when \(\lambda (\mu,\ell)=\lambda(\gamma_n,\ell)\)?
\item[-]\emph{Stability}: What can be said of \(\mu\) when \(\lambda (\mu,\ell)\approx \lambda(\gamma_n,\ell)\)?
\end{itemize}
Looking at the above proof, 
these two questions can usually be reduced to the study of the corresponding ones concerning the optimal map \(T\) in Theorem \ref{thm:caffarelli} (here  $|A|$
denotes the operator norm of a matrix $A$):
\begin{itemize}
\item[-]\emph{Rigidity}: What can be said of \(\mu\) when \(|\nabla T(x)|=1\) for a.e. \(x\) ?
\item[-]\emph{Stability}: What can be said of \(\mu\) when  \(|\nabla T(x)|\approx 1\) (in suitable sense)?
\end{itemize}
Our first main result state that if \(|\nabla T(x)|=1\) for a.e. \(x\) then \(\mu\)  ``splits off'' a Gaussian factor.  More precisely, it splits off as many Gaussian factors as the number of eigenvalues  of \(\nabla T=D^2\varphi\) that are equal to \(1\).
In the following statement and in the sequel, given \(p \in \R^k\)  we denote by \(\gamma_{p,k}\) the Gaussian measure in $\R^k$ with barycenter \(p\), that is,  \(\gamma_{p,k}=(2\pi)^{-k/2}e^{-|x-p|^2/2}dx\).

\begin{theorem}[Rigidity]\label{thm:rigidity}Let \(\gamma_n\) be the Gaussian measure in \(\R^n\), and let \(\mu=e^{-V} dx\) be a probability measure with \(D^2 V\ge \Id_n\). Consider the Brenier map \(T=\nabla \varphi\) from \(\gamma_n\) to \(\mu\), and let 
\[
0\le \lambda_1(D^2 \varphi(x))\le \dots\le \lambda_n(D^2 \varphi(x))\leq 1
\]
 be the eigenvalues of the matrix $D^2\varphi(x)$. If \(\lambda_{n-k+1}(D^2 \varphi(x))=1\) for a.e. \(x\) then
  \(\mu= \gamma_{p,k}\otimes e^{-W(x')}d x'\), where  \(W:\R^{n-k}\to \R\) satisfies \(D^2W\ge \Id_{n-k}\).
\end{theorem}

Our second main result is a quantitative version of the above theorem. Before stating it let us recall that, given two probability measures \(\mu, \nu\in \mathcal P(\R^n)\), the  \(1\)-Wasserstein distance between them is defined as
\[
W_1(\mu,\nu):=\inf\Big\{\int |x-y|\,d\sigma(x,y)\,:\quad \sigma\in \mathcal P(\R^n\times \R^n)\textrm{ such that } ({\rm pr}_1)_\sharp\sigma=\mu,\, ({\rm pr}_2)_\sharp\sigma=\nu\Big\},
\]
where \({\rm pr}_1\) (resp.  \({\rm pr}_2\)) is the projection of \(\R^n\times\R^n\)  onto the first (resp. second) factor.

%
%
%
%

\begin{theorem}[Stability]\label{thm:stability}Let \(\gamma_n\) be the gaussian measure in \(\R^n\)and let \(\mu=e^{-V} dx\) be a probability measure with \(D^2 V\ge \Id_n\). Consider the Brenier map \(T=\nabla \varphi\) from \(\gamma_n\) to \(\mu\), and let 
\[
0\le \lambda_1(D^2 \varphi(x))\le \dots\le \lambda_n(D^2 \varphi(x))\leq 1
\] 
be the eigenvalues of $D^2\varphi(x)$. Let \(\e\in (0,1)\) and assume that 
\begin{equation}\label{eq:almost}
1-\e\le \int \lambda_{n-k+1}(D^2 \varphi(x))\,d\gamma_n(x)\le 1\,.
\end{equation}
Then there exists a probability measure \(\nu= \gamma_{p,k}\otimes e^{-W(x')}d x'\), with  \(W:\R^{n-k}\to \R\) satisying  \(D^2W\ge \Id_{n-k}\), such that
\begin{equation}\label{eq:log}
W_1(\mu,\nu) \lesssim \frac{1}{|\log\e|^{1/4_-}}.
\end{equation}
\end{theorem}
In the above statement, and in the rest of the note,  we are employing the following notation:
\[
X \lesssim Y^{\beta_-} \qquad \textrm{if \(X\le C(n,\alpha)Y^{\alpha}\) for all \(\alpha< \beta\).}
\]
Analogously,
\[
X \gtrsim Y^{\beta_-} \qquad \textrm{if \(C(n,\alpha )X\ge Y^{\alpha}\) for all \(\alpha<\beta\).}
\]

\begin{remark}
We do not expect the stability estimate in the previous theorem to be sharp. In particular, in dimension $1$ an elementary argument (but completely specific to the one dimensional case) gives a linear control in $\e$. Indeed,
if we set $\psi(x):=x^2/2-\varphi(x)$, then our assumption can be rewritten as 
$$
\int \psi'' \,d\gamma_1 \leq \e.
$$
Since $\psi''=(x-T)'>0$, this gives
$$
\int |(x-T)'|\,d\gamma_1 \leq \e
$$
and using the $L^1$-Poincar\'e inequality for the Gaussian measure we obtain
$$
W_1(\mu,\gamma_1) \leq \int|x-y|\,d\sigma_T(x,y)= \int |x-T(x)|\,d\gamma_1(x) \leq C\e,
$$
where $\sigma_T:=(\Id\times T)_\#\gamma_1$.
\end{remark}
%
%

As explained above, Theorems  \ref{thm:rigidity} and  \ref{thm:stability} can be applied to study the structure of \(1\)-log-concave measures (i.e., measures of the form \(e^{-V}dx\) with \(D^2V\ge \Id_n\)) that almost achieve equality in \eqref{eq:lambda2}. To simplify the presentation and emphasize the main ideas, we limit ourselves to a particular instance of \eqref{eq:lambda}, namely the optimal constant in the $L^2$-Poincar\'e inequality for \(\mu\):
\[
\lambda_\mu :=\inf\Bigg\{\frac{\int |\nabla u|^2 \,d\mu} {\int u^2 \,d\mu}\,:\qquad u\in {\rm Lip}(\R^n)\,, \int u \,d\mu=0\Bigg\}.
\]
It is well-known that \(\lambda_{\gamma_n}=1\) and that $\{u_i(x)=x_i\}_{1\leq i \leq n}$ are the corresponding minimizers. In particular it follows by \eqref{eq:lambda2} that, for every \(1\)-log-concave measure \(\mu\),
\begin{equation}\label{poin}
\int u^2\,d\mu \leq \int |\nabla u|^2\,d\mu\qquad\textrm{for all \(u\in {\rm Lip}(\R^n)\) with $\int u\,d\mu=0$. }
\end{equation}
As a consequence of Theorems  \ref{thm:rigidity} and  \ref{thm:stability} we have:

%
%
%

\begin{theorem}
\label{cor:poincare}
Let \(\mu=e^{-V} dx\) be a probability measure with \(D^2 V\ge \Id_n\),
and assume there exist $k$ functions $\{u_i\}_{1\leq i \leq k}\subset W^{1,2}(\R^n,\mu)$, $k \leq n$, such that
$$
\int u_i\,d\mu =0,\qquad  \int u_i^2\,d\mu =1,\qquad \int \nabla u_i\cdot \nabla u_j\, d\mu =0\qquad \forall\,i \neq j,
$$
and
$$
 \int |\nabla u_i|^2\,d\mu\le(1+\e)
$$
for some $\e >0$.
Then there exists a probability measure \(\nu= \gamma_{p,k}\otimes e^{-W(x')}d x'\), with  \(W:\R^{n-k}\to \R\) satisfying  \(D^2W\ge \Id_{n-k}\), such that
$$
W_1(\mu,\nu)\lesssim \frac{1}{|\log\e|^{1/4_-}}.
$$
In particular, if there exist $n$ orthogonal functions  $\{u_i\}_{1\leq i \leq n}$ that attain the equality in \eqref{poin} then $\mu=\gamma_{n,p}$.
\end{theorem}

We conclude this introduction recalling that the  rigidity version of the above theorem (i.e. the case $\e=0$) has already been proved  by Cheng and Zho in \cite[Theorem 2]{CZ} with   completely different techniques.

\section{Proof of Theorem \ref{thm:rigidity}}

\begin{proof}[Proof of Theorem \ref{thm:rigidity}]
Set \(\psi(x):=|x|^2/2-\varphi(x)\) and note that,
as a consequence of Theorem \ref{thm:caffarelli}, \(\psi:\R^n\to \R\) is a  \(C^{1,1}\)  convex function with \(0 \leq D^2\psi \leq \Id\). 
Also, our assumption implies that 
\begin{equation}\label{zeroeigenvalue}
\lambda_1(D^2\psi(x))=\ldots=\lambda_{k}(D^2 \psi(x))=0\qquad \textrm{for a.e. \(x\in \R^d\).}
\end{equation}
We are going to show that $\psi$ depends only on $n-k$ variables. As we shall show later, this will immediately imply the desired conclusion.
In order to prove the above claim, we note it is enough to prove it for \(k=1\), since then one can argue recursively on $\R^{n-1}$ and so on.

Note that \eqref{zeroeigenvalue} implies that 
\begin{equation}
\label{eq:det0}
\det D^2\psi\equiv 0.
\end{equation}
Up to translate \(\mu\) we can subtract a linear function to \(\psi\) and assume  without loss of generality that \(\psi(x)\ge \psi(0)=0\).

Consider the convex set $\Sigma:=\{\psi=0\}$. We claim that $\Sigma$ contains a line.
Indeed, if not, this set would contain an exposed point $\bar x$.
Up to a rotation, we can assume that $\bar x=a\,e_1$ with $a \geq 0$.
Also, since $\bar x$ is an exposed point, 
$$
\Sigma\subset \{x_1\leq a\}\quad
\text{and}
\quad \Sigma\cap \{x_1=a\}=\{\bar x\}.
$$
Hence, by convexity of $\Sigma$, the set $\Sigma\cap\{x_1\geq -1\}$ is compact.

Consider the affine function
$$
\ell_\eta(x):=\eta(x_1+1),\qquad \eta>0\text{ small},
$$
and define $\Sigma_\eta:=\{\psi \leq \ell_\eta\}$.
Note that, as $\eta \to 0$, the sets $\Sigma_\eta$
converge in the Hausdorff distance to the compact set $\Sigma\cap\{x_1\geq -1\}$. In particular,
this implies that $\Sigma_\eta$ is bounded for $\eta$ sufficiently small.

We now apply Alexandrov estimate  (see for instance \cite[Theorem 2.2.4]{figalli_book}) to the convex function $\psi-\ell_\eta$
inside $\Sigma_\eta$,
and it follows by \eqref{eq:det0} that 
(note that $D^2\ell_\eta\equiv 0$) 
$$
|\varphi(x)-\ell_\eta(x)|^n\leq C_{n}({\rm diam}(S_\eta))^n \int_{\Sigma_\eta}\det D^2\psi=0\qquad \forall\,x \in\Sigma_\eta.
$$
In particular this implies that $\psi(0)=\ell_\eta(0)=\eta$,
a contradiction to the fact that $\psi(0)=0.$

Hence, we proved that $\{\psi=0\}$ contains a line, say $\R e_1$.
Consider now a point \(x\in \R^n\). Then, by convexity of $\psi$, 
\[
\psi(x)+\nabla \psi(x)\cdot (s e_1-x)\le \psi(s e_1)=0\qquad \forall\,s \in \R,
\]
and by letting \(s\to \pm\infty\) we deduce that \(\partial_1\psi(x)=\nabla \psi(x)\cdot e_1=0\). Since $x$ was arbitrary, this means that $\partial_1\psi\equiv 0$,
hence \(\psi(x)=\psi(0,x')\), $x'\in \R^{n-1}$.

Going back to $\varphi$, this proves that
$$
T(x)=(x_1,x'-\nabla \psi(x')),
$$
and because $\mu=T_\#\gamma_n$ we immediately deduce that $\mu=\gamma_1\otimes \mu_1$ where $\mu_1:=(\Id_{n-1}-\nabla \psi)_\#\gamma_{n-1}$.

Finally, to deduce that $\mu_1=e^{-W}dx'$ with $D^2W\geq \Id_{n-1}$
we observe that $\mu_1=(\pi')_\#\mu$ where $\pi':\R^n\to\R^{n-1}$ is the projection given by $\pi'(x_1,x'):=x'$.
Hence, the result is a consequence of the fact that $1$-log-concavity is preserved when taking marginals, see \cite[Theorem 4.3]{BL} or \cite[Theorem 3.8]{SW}.
\end{proof}

\section{Proof of Theorem \ref{thm:stability}}

To prove Theorem \ref{thm:stability},
we first recall a basic properties of convex sets
(see for instance \cite[Lemma 2]{Cafbdry} for a proof).

\begin{lemma}
\label{lem:john}
Given $S$ an open bounded convex set in $\R^n$
with barycenter at $0$,
let $\mathcal E$ denote an ellipsoid of minimal volume 
with center $0$ and containing $K$.
Then there exists a dimensional constant $\kappa_n>0$
such that $\kappa_n \mathcal E\subset S$.
\end{lemma}

Thanks to this result, we can prove the following simple geometric lemma:

\begin{lemma}\label{lm:geo}
Let $\kappa_n$ be as in Lemma \ref{lem:john},
set $c_n:=\kappa_n/2$,
and consider \(S\subset \R^n\) an open convex set with barycenter at \(0\).
Assume that \(S\subset B_R\) and  \(\partial S\cap \partial B_R\neq \emptyset\). Then there exists a unit vector $v \in \mathbb S^{n-1}$ such that \(\pm c_n R v\in S\).
 \end{lemma}
\begin{proof}
By scaling we can assume that \(R=1\).

Let $v \in \partial S\cap \partial B_1$, and consider the ellipsoid $\mathcal E$ provided by Lemma \ref{lem:john}.
Since $v \in \overline{\mathcal E}$ and $\mathcal E$ 
is symmetric with respect to the origin, also $-v \in \overline{\mathcal E}$.
Hence
$$
\pm c_n v \in c_n\overline{\mathcal E}\subset \kappa_n\mathcal E
\subset S,
$$
as desired.
\end{proof}

\begin{proof}[Proof of Theorem \ref{thm:stability}]
As in the proof of Theorem \ref{thm:rigidity} we set  \(\psi:=|x|^2/2-\varphi\). Then, inequality  \eqref{eq:almost} gives
\begin{equation}\label{eq:det}
\int \lambda_{k}(D^2\psi)\,d\gamma_n \leq \e.
\end{equation}
Up to subtract a linear function (i.e.  substituting \(\mu\) with one of its translation, which does not affect the conlclusion of the theorem)  we can assume that \(\psi(x)\ge \psi (0)=0\), therefore \(\nabla \psi(0)=\nabla \varphi(0)=0\). Since $(\nabla \varphi)_\#\gamma_n=\mu$ and \(\|D^2\varphi\|_\infty\le 1\), these conditions imply that 
\[
\int |x|\,d\mu(x)=\int|\nabla\varphi(x)|\,d\gamma_n(x)=\int |\nabla \varphi(x)-\nabla \varphi(0)| \,d\gamma_n(x) \le \int |x|\, d\gamma_n(x) \le C_n.
\]
In particular
\[
W_1(\mu,\gamma)\le W_1(\mu,\delta_0)+W_1(\delta_0,\gamma)\le C_n.
\]
This proves that \eqref{eq:log} holds true with \(\nu=\gamma_n\) and with a constant \(C\approx |\log \e_0|^{1/4}\) whenever \(\e\ge \e_0\). Hence, when showing the validity of \eqref{eq:log}, we can safely assume that \(\e\le \e_0(n)\ll1\).
Furthermore, we can assume that the graph of \(\psi\) does not contain lines (otherwise, by the proof of Theorem \ref{thm:rigidity}, we would deduce that $\mu$ splits a Gaussian factor, and we could simply repeat the argument in $\R^{n-1}$).

Thanks to these considerations, we can apply  \cite[Lemma 1]{Cafbdry}  to find a slope \(p\in \R^n\) such that the open convex set
\[
S_1:=\{x\in \R^n: \psi (x)< p\cdot x+ 1\}
\]
 is  nonempty, bounded,  and with barycenter at $0$. Applying the Aleksandrov estimate in \cite[Theorem 2.2.4]{figalli_book} to the convex function \(\tilde\psi(x):=\psi(x)-p\cdot x-1\) inside the set \(S_1\), we get (note that $D^2\tilde\psi=D^2\psi$)
 \begin{equation}\label{eq:abp}
 1\le \Bigl(-\min_{S_1} \tilde\psi \Bigr)^n\le C_n(\diam(S_1))^n\int_{S_1} \det D^2 \psi.
 \end{equation}
Consider now the smallest radius \(R>0\) such that \(S_1\subset B_R\) (note that  \(R<+\infty\) since \(S_1\) is bounded). 
Since $\gamma_n \geq c_ne^{-R^2/2}$ in $B_R$
and $\lambda_i(D^2\psi)\leq 1$ for all $i=1,\ldots,n$,
\eqref{eq:det} implies that
$$
\int_{B_R}\det D^2\psi \leq C_n e^{R^2/2}\e.
$$
Hence, using \eqref{eq:abp}, since ${\rm diam}(S_1)\leq 2R$ we get
$$
1 \leq C_n R^n e^{R^2/2}\e
$$ 
which yields
\begin{equation}\label{eq:R}
R \gtrsim |\log  \e|^{1/2_+}.
\end{equation}
Now, up to a rotation and by Lemma \ref{lm:geo}, we can assume that 
\[
\pm c_nRe_1 \in S_1.
\]
Consider $1\ll \rho \ll R^{1/2}$ to be chosen. Since \(S_1\subset B_R\) and \(\psi \ge 0\) we get that \(|p|\le 1/R\), therefore  \(\psi\le 2\) on \(S_1\subset B_R\). Hence
$$
2 \geq \psi(z)\geq \psi(x)+\langle \n \psi(x), z-x\rangle \geq \langle \n \psi(x), z-x\rangle \qquad \forall\,z \in S_1,\,x \in B_{\rho}.
$$
Thus, since $|\n \psi|\leq \rho$ in $B_{\rho}$ (by $\|D^2\psi\|_{L^\infty(\R^n)}\le 1$ and \(|\nabla \psi(0)|=0\)),
choosing $z =\pm c_nRe_1$ we get
\begin{equation}\label{eq:der}
|\partial_1\psi|\leq \frac{C_n\rho^2}{R} \qquad \text{inside }B_{\rho}.
\end{equation}
Consider now $\bar x_1\in [-1,1]$ (to be fixed later) and define $\psi_1(x'):=\psi(\bar x_1,x')$ with $x' \in \R^{n-1}$. Integrating \eqref{eq:der} with respect to $x_1$ inside $B_{\rho/2}$, we get 
$$
|\psi -\psi_1| \leq C_n \frac{\rho^3}{R} \qquad \text{inside }B_{\rho/2}.
$$
Thus, using the interpolation inequality
$$
\|\n \psi -\n \psi_1\|_{L^\infty(B_{\rho/4})}^2 \leq  C_n\| \psi -\psi_1\|_{L^\infty(B_{\rho/2})} \|D^2 \psi -D^2 \psi_1\|_{L^\infty(B_{\rho/2})}
$$
and recalling that $\|D^2\psi\|_{L^\infty(\R^n)} \leq 1$ (hence $\|D^2\psi_1\|_{L^\infty(\R^{n-1})}\leq 1$),
we get
\begin{equation*}
|\n \psi -\n \psi_1| \leq C_n \frac{\rho^{3/2}}{R^{1/2}} \qquad \text{inside }B_{\rho/4}.
\end{equation*}
If $k=1$ we stop here, otherwise we notice that \eqref{eq:det} implies that 
$$
\int_\R d\gamma_1(x_1)\int_{\R^{n-1}} 
\det D_{x'x'}^2\psi(x_1,x')\,d\gamma_{n-1}(x') \leq
\int_\R d\gamma_1(x_1)\int_{\R^{n-1}} 
\lambda_{2}(D^2\psi)(x_1,x')\,d\gamma_{n-1}(x') \leq \e,
$$
where we used  that\footnote{This inequality follows from the general fact that, given $A \in \R^{n\times n}$ symmetric matrix and \(W\subset \R^n\) a \(k\)-dimensional vector space,
\[
\lambda_1\big(A\big|_W\big)=\min_{v\in W} \frac{A v\cdot v}{|v|^2}\le \max_{\substack{v\in W'\subset \R^n\\  \textrm{\(W'\) \(k\)-dim}}} \min_{W'} \frac{A v\cdot v}{|v|^2}=\lambda_{n-k+1}(A).
\]}
$$
\lambda_1\bigl(D^2\psi|_{\{0\}\times\R^{n-1}}\bigr) \leq \lambda_2(D^2\psi)
$$
and that (since $D^2\psi\leq \Id$)
$$
\det D_{x'x'}^2\psi(x_1,x')\leq \lambda_1\bigl(D^2\psi|_{\{0\}\times\R^{n-1}}\bigr).
$$
Hence, by Fubini's Theorem, there exists $\bar x_1\in [-1,1]$ such that $\psi_1(x')=\psi(\bar x_1,x')$ satisfies
$$
\int_{\R^{n-1}} \det D^2\psi_1\,d\gamma_{n-1}(x) \leq C_n\e.
$$
This allows us to repeat the argument above in $\R^{n-1}$ with 
$$
\widetilde \psi_1(x'):=\psi_1(x')-\nabla_{x'} \psi_1(0)\cdot x'-\psi_1(0)
$$ 
in place of $\psi$, and up to a rotation we deduce that
$$
|\n \widetilde{\psi}_1 -\n \psi_2| \leq C_n \frac{\rho^{3/2}}{R^{1/2}} \qquad \text{inside }B_{\rho/4}.
$$
where $\psi_2(x''):=\psi_1(\bar x_2,x'')$, where $\bar x_2 \in [-1,1]$ is arbitrary.
By triangle inequality, this yields 
$$
|\n \psi +p'- \n \psi_2| \leq C_n \frac{\rho^{3/2}}{R^{1/2}} \qquad \text{inside }B_{\rho/4},
$$
where  \(p'=-(0,\nabla_{x'} \psi(\bar x_1,0)\).
Note that, since \(|\bar x_1|\le 1\), \(\nabla \psi(0)=0\), and \(\|D^2\psi\|_\infty\le 1\), we have \(|p|\le 1\). Iterating this argument $k$ times, we conclude that
$$
|\n \psi +\bar p-\n \psi_k| \leq C_n \frac{\rho^{3/2}}{R^{1/2}} \qquad \text{inside }B_{\rho/4}
$$
where \(\bar p=(p,p'')\in \R^k\times \R^{n-k}=\R^n\) with \(|\bar p|\le C_n\),
$$
\psi_k(y):=\psi(\bar x_1,\ldots,\bar x_k,y),\qquad y \in \R^{n-k},
$$
and $\bar x_i \in [-1,1]$.
Recalling that $\n \varphi=x-\nabla \psi$, we have proved that
$$
T(x)=\n \varphi(x)=(x_1+p_1,\ldots,x_k+p_k,S(y)+p'') +Q(x),
$$
where $Q:=-(\nabla \psi-\nabla\psi_k+\bar p)$ satisfies
\[
\|Q\|_{L^\infty(B_{\rho})}\leq C_n\frac{\rho^{3/2}}{R^{1/2}}\qquad\textrm{and}\qquad |Q(x)|\le C_n (1+|x|)
\]
(in the second bound we used that $T(0)=\nabla \varphi(0)=0$, \(|p|\le C_n\), and 
$T$ is $1$-Lipschitz).
Hence, if we set $\nu:=(S+p'')_\# \gamma_{n-k}$, we have 
$$
W_1(\mu,\gamma_{p,k}\otimes \nu) \leq \int |Q| \,d\gamma_n
\leq C_n\frac{\rho^{3/2}}{R^{1/2}}+C_n\int_{\R^n\setminus B_{\rho}} |x| \,d\gamma_n 
= C_n\frac{\rho^{3/2}}{R^{1/2}}+C_n\rho^{n}e^{-\rho^2/2},
$$
so, by choosing $\rho:=(\log R)^{1/2}$, we get
$$
W_1(\mu,\gamma_{p,k}\otimes \nu) \lesssim\frac{1}{R^{1/2_-}}.
$$
Consider now $\pi_k:\R^n\to\R^n$ and $\bar \pi_{n-k}:\R^n\to \R^{n-k}$
the orthogonal projection onto the first $k$ and the last $n-k$ coordinates, respectively.
Define $\mu_1:=(\pi_k)_\# (e^{-V}dx)$, $\mu_2:=(\bar \pi_{n-k})_\# (e^{-V}dx)$, and note that these are \(1\)-log-concave measures in \(\R^k\) and \(\R^{n-k}\) respectively
(see  \cite[Theorem 4.3]{BL} or \cite[Theorem 3.8]{SW}). In particular \(\mu_2=e^{-W}\) with \(D^2W\ge \Id_{n-k}\). Moreover, since $W_1$ decreases under orthogonal projection,
$$
W_1(\mu_2,\nu)=W_1\bigl((\bar \pi_{n-k})_\#\mu, (\bar \pi_{n-k})_\#(\gamma_{p,k}\otimes\nu) \bigr) \leq W_1(\mu,\gamma_{p,k}\otimes \nu) \lesssim\frac{1}{R^{1/2_-}},
$$
thus
\[
\begin{split}
W_1(\mu,\gamma_{p,k}\otimes \mu_2)
&\leq W_1(\mu,\gamma_{p,k}\otimes \nu)+W_1(\gamma_{p,k}\otimes \nu,\gamma_{p,k}\otimes \mu_2)
\\
&\le  W_1(\mu,\gamma_{p,k}\otimes \nu)+W_1( \nu, \mu_2)  \lesssim\frac{1}{R^{1/2_-}}
\end{split}
\]
where we used the elementary fact that \(W_1(\gamma_{p,k}\otimes \nu,\gamma_{p,k}\otimes \mu_2)\le W_1( \nu, \mu_2) \). Recalling \eqref{eq:R}, this proves that
$$
W_1(\mu,\gamma_{p,k}\otimes \mu_2) \lesssim\frac{1} {|\log \e|^{1/4_-}},
$$
concluding the proof.
\end{proof}

\section{Proof of Theorem \ref{cor:poincare}}
\begin{proof}[Proof of Theorem \ref{cor:poincare}]
As in the proof of Theorem \ref{thm:stability}, it is enough to prove the result when $\e \leq \e_0 \ll 1$.

Let $\{u_i\}_{1\leq i \leq k}$ be as in the statement,
and set $v_i:=u_i\circ T$, where $T=\nabla \varphi:\R^n\to \R^n$ is the Brenier map
from $\gamma_n$ to $\mu$.
Note that since $T_\#\gamma_n=\mu$,
$$
\int v_i\,d\gamma_n=\int u_i\circ T\,d\gamma_n=\int u_i\,d\mu=0.
$$
Also, since $|\nabla T|\leq 1$ and by our assumption on $u_i$,
\begin{align*}
\int |\nabla v_i|^2\,d\gamma_n&\leq \int |\nabla u_i|^2\circ T\,d\gamma_n
=\int |\nabla u_i|^2\,d\mu\\
&\le  (1+\e) 
\int u^2_i\,d\mu =(1+\e) \int v^2_i\,d\gamma_n
\le (1+\e) \int |\nabla v_i|^2\,d\gamma_n,
\end{align*}
where the last inequality follows from the Poincar\'e inequality for $\gamma_n$
applied to $v_i$. Since  
$$
\int |\nabla u_i|^2\,d\mu\le (1+\e),
$$
 this proves that
\begin{equation}
\label{eq:eps 1}
0 \leq \int \Bigl( |\n u_i|^2\circ T - |\n v_i|^2\Bigr)\,d\gamma_n \leq \e \int |\nabla v_i|^2\,d\mu \leq \e (1+\e).
\end{equation}
Moreover, by Theorem \ref{thm:caffarelli},  $\nabla T=D^2\varphi $ is a symmetric matrix satisfying $0\leq \n T\leq \Id_n$,
therefore  \((\Id-\nabla T)^2\le \Id -(\n T)^2\). Hence,
since $\n v_i= \nabla T\cdot\n u_i\circ T $, it follows by \eqref{eq:eps 1} that 
\begin{equation}\label{eq:v1}
\begin{aligned}
\int | \n u_i\circ T-\n v_i|^2\,d\gamma_n
&=\int |(\Id_n - \nabla T)\cdot \n u_i\circ T|^2\,d\gamma_n\\
&=\int (\Id_n - (\n T))^2[\n u_i\circ T,\nabla u_i\circ T]\,d\gamma_n\\
&\le \int (\Id_n - (\n T)^2)[\n u_i\circ T,\nabla u_i\circ T]\,d\gamma_n\\
&=\int \Bigl( |\n u_i|^2\circ T - |\n v_i|^2\Bigr)\,d\gamma_n\leq 2\e,
\end{aligned}
\end{equation}
where, given a matrix \(A\) and a vector \(v\), we have used the notation \(A[v,v]\) for \(Av\cdot v\).
In particular, recalling the orthogonality constraint $\int \n u_i\cdot \n u_j\,d\mu =0$, we deduce that
\begin{equation}
\label{eq:ortho eps}
\int \n v_i\cdot \n v_j\,d\gamma_n=O(\sqrt{\e}).
\end{equation}
In addition, if we set 
$$
f_i(x):=\frac{\n u_i\circ T(x)}{|\n u_i\circ T(x)|}
$$
then, using again that \(|\nabla T|\le 1\),
\begin{equation}
\label{eq:2eps}
\int |\n (u_i\circ T)|^2\Bigl(1 - |\n T\cdot f_i|^2\Bigr)\,d\gamma
\leq \int |\n u_i|^2\circ T\Bigl(1 - |\n T\cdot f_i|^2\Bigr)\,d\gamma_n \leq 2\e.
\end{equation}
Now, for \(j\in \N\), let \(H_j:\R\to \R\) be the one dimensional  Hermite polynomial of degree \(j\) (see \cite[Section 9.2]{DaPrato} for a precise definition). It is well known (see for instance \cite{DaPrato}) that for \(J=(j_1,\dots,j_n)\in \N^n\) the functions
\[
H_J(x_1,\dots,x_n)=H_{j_1}(x_1)H_{j_2}(x_2)\cdot\dots\cdot H_{j_n}(x_n)
\]
form a Hilbert basis of \(L^2(\R^n,\gamma_n)\). Hence,
since \(\alpha^i_0=\int v_i\,d\gamma_n=0\), we can write 
\[
v_i=\sum_{J\in \N^n\setminus \{0\}} \alpha^i_J H_J.
\]
By some elementary properties of Hermite polynomials (see \cite[Proposition 9.3]{DaPrato}), we get 
\[
1=\int v_i^2 d\gamma_n=\sum_{J\in \N^n\setminus \{0\}}\big(\alpha^i_J\big)^2, \qquad \int |\nabla v_i|^2d\gamma _n=\sum_{J\in \N^n\setminus \{0\}} |J| \big(\alpha^i_J\big)^2.
\]
Hence, combining the above equations with the bound \(\int |\nabla v_i|^2d\gamma _n\le (1+\e)\), we obtain
\[
\e\ge \int |\nabla v_i|^2d\gamma _n-\int  v_i^2d\gamma _n=\sum_{J\in \N^n\,, |J|\ge 2}(|J|-1)\big(\alpha^i_J\big)^2\ge \frac{1}{2}\sum_{J\in \N^n\,, |J|\ge 2}|J|\big(\alpha^i_J\big)^2,
\]
where $|J|=\sum_{m=1}^nj_m$.
Recalling that the first Hermite polynomials are just linear functions (since \(H_1(t)=t\)), using the notation
$$
\alpha_j^i:=\alpha_J^i \qquad \text{with }J=e_j \in \mathbb N^n
$$
we deduce that
\[
v_i(x)=\sum_{j=1}^n \alpha^{i}_j x_j+z(x),
\qquad
\textrm{with} 
\qquad
\|z\|^2_{W^{1,2}(\R^n,\gamma_n)}=O(\e).
\]
In particular, if we define the vector
\[
V_i:=\sum_{j=1}^n \alpha^i_{j}e_j \in \R^n,
\]
and we recall that  $\int |\nabla v_i|^2\,d\gamma_n=1+O(\e)$  and the almost orthogonality relation \eqref{eq:ortho eps}, we infer that $|V_i|=1+O(\e)$ and $|V_i\cdot V_l|=O(\sqrt{\e})$ for all $i \neq l\in\{1,\dots,k\}$.

Hence, up to a rotation, we can assume that $|V_i-e_i|=O(\sqrt{\e})$ for all $i=1,\ldots,k$, and \eqref{eq:v1} yields
\begin{equation}
\label{eq:close 1}
\int |\n (u_i \circ T) -e_i|^2\,d\gamma_n \leq C\,\e.
\end{equation}
Since  $0 \leq 1 - |\n T\cdot f_i|^2\leq 1$,
it follows by \eqref{eq:2eps} and \eqref{eq:close 1} that
\begin{equation}
\label{eq:close 2}
\int \Bigl(1 - |\n T\cdot f_i|^2\Bigr)\,d\gamma_n \leq
2\int \Bigl(|\n (u_i\circ T)|^2+|\n (u_i \circ T) -e_i|^2\Bigr)\Bigl(1 - |\n T\cdot f_i|^2\Bigr)\,d\gamma_n
\leq C \e.
\end{equation}
Set $w_i:=\n u_i \circ T$ so that $f_i=\frac{w_i}{|w_i|}$.
We note that, since all the eigenvalues of \(\nabla T=D^2\varphi\) are bounded by \(1\), given $\delta \ll 1$ the following holds: whenever
$$
|\n T\cdot w_i -e_i|\leq \delta
\qquad\text{and}\qquad
|\n T\cdot f_i|\geq 1-\delta
$$
then $|w_i|=1+O(\delta)$. In particular,
$$
|\n T\cdot f_i - e_i|\leq C\delta.
$$
Hence, if $\delta \leq \delta_0$ where \(\delta_0\) is a small geometric constant, this implies that
the vectors $f_i$ are a basis of $\R^k$, and 
$$
\nabla T|_{{\rm span}(f_1,\ldots,f_k)} \geq (1-C\delta)\,\Id.
$$
Defining $\psi(x):=|x|^2/2-\varphi(x)$,
this proves that
\begin{equation}\label{inclusione}
\biggl\{x\,:\,\sum_i |\n T(x)\cdot w_i(x) -e_i| + \Bigl(1-|\n T(x)\cdot f_i(x)|\Bigr)\leq \delta\biggr\}
\subset \left\{x\,:\,\lambda_{n-k+1}(D^2\psi(x)) \leq C\delta\right\}
\end{equation}
for all $0<\delta\leq \delta_0$.
By the layer-cake formula, \eqref{eq:close 1}, and \eqref{eq:close 2},
this implies that
\begin{equation*}
\begin{split}
\int_{\{\lambda_{n-k+1}(D^2\psi) \leq C\delta_0\}}\lambda_{n-k+1}(D^2\psi)\,d\gamma_n
&= C\int_0^{\delta_0} \gamma_n\bigl(\{\lambda_{n-k+1}(D^2\psi) > Cs\}\bigr)\,ds
\\
&\leq C\sum_i \int_0^{\delta_0} \gamma_n\bigl(\{|\n T\cdot w_i - e_i|>s\}\bigr)\,ds\\
&\quad+C\sum_i \int_0^{\delta_0} \gamma_n\bigl(\{1-|\n T\cdot f_i|>s\}\bigr)\,ds
\\
&\leq C \sum_i \int \Bigl(|\n T\cdot w_i - e_i|+\bigl(1-|\n T\cdot f_i|\bigr)\Bigr)\,d\gamma_n\leq C\sqrt{\e}.
\end{split}
\end{equation*}
On the other hand, again by \eqref{inclusione},  \eqref{eq:close 1}, \eqref{eq:close 2},
and Chebishev's inequality, 
\begin{multline*}
\gamma_n\bigl(\{\lambda_{n-k+1}(D^2\psi) > C\delta_0\}\bigr)
\leq \sum_i \gamma_n\bigl(\{|\n T\cdot w_i - e_i|>\delta_0\}\bigr)\\
\quad+\sum_i\gamma_n\bigl(\{1-|\n T\cdot f_i|>\delta_0\}\bigr)\leq C\,\frac{\e}{\delta_0^2}.
\end{multline*}
Hence, since $\delta_0$ is a small but fixed geometric constant, combining the two equations above and recalling that \(\lambda_{n-k+1}(D^2\psi)\le 1\), we obtain
$$
\int \lambda_{n-k+1}(D^2\psi)\,d\gamma_n \leq C\sqrt{\e}.
$$
This implies that \eqref{eq:almost} holds with $C\sqrt{\e}$ in place of $\e$, and the result follows by Theorem \ref{thm:stability}.
\end{proof}

\subsection*{Acknowledgements}
G.D.P. is supported by the MIUR SIR-grant ``Geometric Variational Problems'' (RBSI14RVEZ). G.D.P is a  member of the ``Gruppo Nazionale per l'Analisi Matematica, la Probabilit\`a e le loro Applicazioni'' (GNAMPA) of the Istituto Nazionale di Alta Matematica (INdAM). A.F. is supported by NSF Grants DMS-1262411 and 
DMS-1361122.

\end{document}